\documentclass{amsart}
\usepackage[ocgcolorlinks,linktoc=all]{hyperref}
\usepackage{cite}
\hypersetup{citecolor=blue,linkcolor=red}
\newtheorem{theorem}{Theorem}

\newtheorem*{claim}{Claim}

\newtheorem*{thmmain}{Theorem}
\newtheorem{lemma}[theorem]{Lemma}
\newtheorem{proposition}[theorem]{Proposition}

\newtheorem{corollary}[theorem]{Corollary}
\theoremstyle{definition}

\theoremstyle{remark}
\newtheorem{remark}[theorem]{Remark}

\numberwithin{equation}{section}

\begin{document}

\title[]
 {A unified flow approach to smooth, even $L_p$-Minkowski problems}

\author[P. Bryan]{Paul Bryan}
\address{Department of Mathematics, Macquarie University
NSW 2109, Australia}
\email{paul.bryan@mq.edu.au}
\author[M.N. Ivaki]{Mohammad N. Ivaki}
\address{Department of Mathematics, University of Toronto, Ontario,
M5S 2E4, Canada}
\email{m.ivaki@utoronto.ca}
\author[J. Scheuer]{Julian Scheuer}
\address{Albert-Ludwigs-Universit\"{a}t,
Mathematisches Institut, Ernst-Zermelo-Str. 1, 79104
Freiburg, Germany}
\email{julian.scheuer@math.uni-freiburg.de}
\dedicatory{}
\subjclass[2010]{}

\begin{abstract}
We study long-time existence and asymptotic behavior for a class of anisotropic, expanding curvature flows. For this we adapt new curvature estimates, which were developed by Guan, Ren and Wang to treat some stationary prescribed curvature problems. As an application we give a unified flow approach to the existence of smooth, even $L_p$-Minkowski problems in $\mathbb{R}^{n+1}$ for $p>-n-1.$
\end{abstract}

\maketitle
\section{Introduction}
Consider a smooth, closed, strictly convex hypersurface $M_0$ in Euclidean space $\mathbb{R}^{n+1}$, $n\ge 2,$ given by a smooth embedding $F_0:M\to \mathbb{R}^{n+1}.$ Suppose the origin is in the interior of the region enclosed by $M_0.$ We study the long-time behavior of a family of hypersurfaces $\{M_t\}$ given by smooth maps $F:M\times [0,T)\to \mathbb{R}^{n+1}$ satisfying the initial value problem
\begin{equation}\label{e: flow0}
 \partial_{t}F(x,t)=\varphi(\nu(x,t))\frac{(F(x,t)\cdot \nu (x,t))^{2-p}}{\mathcal{K}(x,t)} \nu(x,t),\quad
 F(\cdot,0)=F_{0}(\cdot).
\end{equation}
Here $\mathcal{K}(\cdot,t)$ and $\nu(\cdot,t)$ are the Gauss curvature and the outer unit normal vector of $M_t=F(M,t)$ and $\varphi$ is a positive, smooth function on $\mathbb{S}^{n}$. Furthermore, $T$ is the maximal time for which the solution exists.

For $p=2,~\varphi\equiv1$, flow (\ref{e: flow0}) was studied by Schn\"{u}rer \cite{Oliver 2006} in $\mathbb{R}^3$ and by Gerhardt \cite{Gerhardt 2014} in higher dimensions. Both works rely on the reflection principle of Chow and Gulliver \cite{Bennett Chow and Robert Gulliver 1996} and McCoy \cite{James A. McCoy 2003}. Their result is as follows: the volume-normalized flow evolves any $M_0$ in the $C^{\infty}$-topology to an origin-centered sphere. For $p>2$ and $\varphi\equiv1$  it follows from Chow-Gulliver \cite[Theorem 3.1]{Bennett Chow and Robert Gulliver 1996} (see also Tsai \cite[Example 1]{Tsai 2005}) that (\ref{e: flow0}) evolves $M_0$, after rescaling to fixed volume, in the $C^{1}$-topology to an origin-centered sphere. We refer the reader to the paper \cite{Ivaki2014-gauss} regarding a rather comprehensive list of previous works on this curvature flow. In particular, in either case $\varphi\neq 1$ or $ \varphi\equiv1, -n-1<p<2$, we are not aware of any result in the literature on the asymptotic behavior of the flow. The following theorem was proved in \cite{Ivaki2014-gauss} regarding the case $p=-n-1,~\varphi\equiv1$ (in this case the flow belongs to a family of centro-affine normal flows introduced by Stancu in \cite{Alina 2012}).

Let us set write $B$ for the unit ball of $\mathbb{R}^{n+1}$ and put
\[\tilde{K_t}:=(V(B)/V(K_t))^{1/(n+1)}K_t,\] where $K_t$ denotes the convex body enclosed by $M_t$ and $V(\cdot)$ is the $(n+1)$-dimensional Lebesgue measure.

\begin{thmmain}[\cite{Ivaki2014-gauss}]
Let $n\ge2,$ $p=-n-1,~\varphi\equiv1$ and suppose $K_0$ has its Santal\'{o} point at the origin, i.e.,
$$\int_{\mathbb{S}^{n}}\frac{u}{h_{K_0}(u)^{n+2}}d\sigma(u)=0.$$ Then there exists a unique solution $\{M_t\}$ of flow (\ref{e: flow0}), such that $\tilde{M}_t$ converges in $C^{\infty}$ to an origin-centered ellipsoid.
\end{thmmain}

Here $h_{K_0}$ is the support function of $K_0.$ A closed, convex hypersurface $M_0$ can be described in terms of its support function $h_{K_0}:\mathbb{S}^n\to\mathbb{R}$ defined by
\[h_{K_0}(u)=\sup\{u\cdot x: x\in M_0\}.\]
If $M_0$ is smooth and strictly convex, then $h_{K_0}(u)=u\cdot F_0(\nu^{-1}(u)).$

From the evolution equation of $F(\cdot,t)$ it follows that $$h(\cdot,t):=h_{K_t}(\cdot):\mathbb{S}^{n}\times [0,T)\to \mathbb{R}$$ evolves by
\begin{equation}\label{eq: flow4}
\partial_th(u,t)=\varphi(u)(h^{2-p}S_{n})(u,t),
\end{equation}
where $S_{n}(u,t)=1/\mathcal{K}(\nu^{-1}(u,t),t).$
A homothetic self-similar solution of this flow satisfies
\begin{align} \label{def: self similar}
h^{1-p} \det (\bar{\nabla}^2 h + \operatorname{Id}h)=\frac{c}{\varphi},
\end{align}
for some positive constant $c.$ Here $\bar{\nabla}$ is the covariant derivative on $\mathbb{S}^{n}$. Note that $S_{n}=\det (\bar{\nabla}^2 h + \operatorname{Id}h).$

We list the main results of the paper extending the previous mentioned results.

\begin{theorem}\label{thm1}
Let $-n-1<p<\infty$ and $\varphi$ be a positive, smooth even function on $\mathbb{S}^{n}$ i.e., $\varphi(u)=\varphi(-u)$. Suppose $K_0$ is origin-symmetric. There exists a unique origin-symmetric solution $\{M_t\}$ of (\ref{e: flow0}) such that $\{\tilde{M}_t\}$ converges for a subsequence of times in $C^{1}$ to a smooth, origin-symmetric, strictly convex solution of (\ref{def: self similar}). Also, when $p\leq n+1$ the convergence is in $C^{\infty}$, and if $p\ge 1$ the convergence holds for the full sequence.
\end{theorem}
If $-n-1< p\leq -n$, we can extend the result of the previous theorem by dropping the assumption that $\varphi$ is even. 
\begin{theorem}
Let $-n-1< p\leq -n$ and $K_0$ satisfy $$\int_{\mathbb{S}^{n}}\frac{u}{\varphi(u)h_{K_0}(u)^{1-p}}d\sigma(u)=0.$$ There exists a unique solution $\{M_t\}$ of flow (\ref{e: flow0}) such that $\{\tilde{M}_t\}$ converges for a subsequence of times in $C^{\infty}$ to a positive, smooth, strictly convex solution of (\ref{def: self similar}).
\end{theorem}
Given any convex body $K_0$, there exists a vector $\vec{v}$ such that $K_0+\vec{v}$ has the origin in its interior and it satisfies the assumption of the second theorem.

For $\varphi\equiv 1$ we prove the following theorem.
\begin{theorem}
Let $1\neq p>-n-1,~\varphi\equiv 1$ and $K_0$ satisfy $$\int_{\mathbb{S}^{n}}\frac{u}{h_{K_0}(u)^{1-p}}d\sigma(u)=0.$$ Then there exists a unique solution $\{M_t\}$ of (\ref{e: flow0}) such that $\{\tilde{M}_t\}$ converges  in $C^{1}$ to the unit sphere. In addition, for $1\neq p\leq n+1$ the convergence holds in $C^{\infty}$.
\end{theorem}
For $p\neq n+1$, self-similar solutions to \eqref{e: flow0} are solutions of the $L_p$-Minkowski problem \eqref{e:lp}, and for $p=n+1$, a self-similar solution to \eqref{e: flow0} is a solution to the normalized $L_{n+1}$-Minkowski problem \eqref{e:nlp}, which we shall introduce them now.

The Minkowski problem deals with existence, uniqueness, regularity, and stability of closed convex hypersurfaces whose Gauss curvature (as a function of the outer normals) is preassigned. Major contributions to this problem were made by Minkowski \cite{M1,M2}, Aleksandrov \cite{A2,A3,A4}, Fenchel and Jessen \cite{FJ}, Lewy \cite{Le1,Le2}, Nirenberg \cite{N}, Calabi \cite{Cal}, Pogorelov \cite{P1,P2}, Cheng and Yau \cite{ChYau}, Caffarelli, Nirenberg, and Spruck \cite{CNS}, and others. A generalization of the Minkowski problem known as the $L_p$-Minkowski problem was introduced by Lutwak in \cite{Lu1}, where for any $1<p\neq n+1$ and a preassigned even Borel measure on $\mathbb{S}^n$ whose support does not lie in a great sphere of $\mathbb{S}^{n}$ the existence and uniqueness of the solution were proved. This generalization for $1<p\neq n+1$ was further studied by Lutwak and Oliker in \cite{LuO}, where they obtained the $C^{k,\alpha}$ regularity of the solution. Solutions to many cases of these generalized problems followed later in \cite{1,AnCrys,AndrewsBen2000,Andrews 2003,27,36,39,49,50,51,72,LYZ,79,104,110,s1,s2,Zhu1,Zhu2,Zhu3,jiang,QL,DZ}.

For $p\neq n+1$, in the smooth category, the $L_p$-Minkowski problem asks, given a smooth, positive function $\varphi : \mathbb{S}^n \to \mathbb{R}$, does there exist a smooth, closed, strictly convex hypersurface $M_0\subset \mathbb{R}^{n+1}$ such that
\begin{equation}
\label{e:lp}
\frac{h^{1-p}(\nu(x))}{\mathcal{K}(x)} =\frac{1}{\varphi(\nu(x))}
\end{equation}
where $x \in M_0$, $h$ denotes the support function, $\mathcal{K}$ the Gauss curvature and $\nu$ the Gauss map $M_0 \to \mathbb{S}^n$. The \emph{even} $L_p$-Minkowski problem requires in addition, that $\varphi$ is an even function. The case $p=1$ is the original Minkowski problem.

The special case of $p=n+1$ is troubling since (\ref{e:lp}) might not have a solution. To remedy this, Lutwak, Yang and Zhang introduced a normalized formulation of the $L_{n+1}$-Minkowski problem in \cite{LYZ} and they proved the existence and uniqueness of the solution for any prescribed even Borel measure on $\mathbb{S}^n$ whose support is not contained in a great sphere of $\mathbb{S}^{n}$. In the smooth category, the normalized $L_{n+1}$-Minkowski problem asks for the existence of a smooth, closed, strictly convex hypersurface $M_0\subset \mathbb{R}^{n+1}$ that solves
\begin{equation}
\label{e:nlp}
\frac{1}{h^{n}(\nu(x))\mathcal{K}(x)} =\frac{V(K_0)}{\varphi(\nu(x))},
\end{equation}
where $K_0$ is the convex body with the boundary $M_0.$
In the rest of the paper, the $L_p$-Minkowski problem  refers to either (\ref{e:lp}) or (\ref{e:nlp}), and we avoid the word ``normalized".

The existence and regularity of solutions to the $L_p$-Minkowski problem are rather comprehensively discussed in \cite{39} for $p >-n-1$. Our study on (\ref{e: flow0}) provides an alternative variational treatment (based on curvature flow) of the even $L_p$-Minkowski problem. For $p=1$, Chou-Wang \cite{Chou Wang 2000} treated the classical $L_1$-Minkowski problem in the smooth category by a logarithmic Gauss curvature flow. For $n=1$, and $1\neq p>-3$, the existence of solutions to the $L_p$-Minkowski problems follows from Andrews' results \cite{Andrews1998} on the asymptotic behavior of a family of contracting and expanding flows of curves. Also, in higher dimensions, the existence of solutions to the $L_p$-Minkowski problems follows from \cite{AndrewsBen2000} when $-n-1<p\leq -n+1$ (a short proof of this is also given in \cite{Ivaki-Proc}) or when $\varphi$ is even (e.q., $\varphi(u)=\varphi(-u)$) and $-n+1<p<1.$ See also \cite{Andrews1999,AGN, GN, U1, U2}.

Using our results for the flows above, it is now a simple matter to give a new, unified proof of the smooth, even $L_p$-Minkowski problem for all ranges of $p>-n-1$.

\begin{corollary}
Let $-n-1<p<\infty$ and $\varphi$ be a positive, smooth even function on $\mathbb{S}^{n}$ i.e., $\varphi(u)=\varphi(-u)$. Then for $p\neq n+1$ there exists an origin-symmetric, smooth, strictly convex body such that (\ref{e:lp}) is satisfied. For $p=n+1$, there exists an origin-symmetric, smooth, strictly convex body such that (\ref{e:nlp}) is satisfied.
\end{corollary}
\begin{proof}
By the first part of Theorem \ref{thm1} (only the convergence for a \emph{subsequence} of times is needed), there exists a smooth, strictly convex body $K$ with the volume of the unit ball and a constant $c>0$ such that
  \[\frac{h}{\mathcal{K}}=\frac{ch^{p}}{\varphi}.\]
  Hence $c\int_{\mathbb{S}^{n}}\frac{ h^{p}}{\varphi}d\sigma=(n+1)V(B^n).$ Thus there is a solution to
  \[\frac{h^{1-p}(\nu(x))}{\mathcal{K}(x)}=\left(\frac{(n+1)V(B)}{\int_{\mathbb{S}^{n}}\frac{ h^{p}}{\varphi}d\sigma}\right)\frac{1}{\varphi(\nu(x))}.\]
Now let us define
\[\lambda:=\left\{
  \begin{array}{ll}
    \left(\frac{\int_{\mathbb{S}^n}\frac{h^p}{\varphi}d\sigma}{(n+1)V(B)}\right)^{\frac{1}{n+1-p}}, & \hbox{$p\neq n+1$;} \\
    \left(\frac{(n+1)V(B)}{V(K)\int_{\mathbb{S}^n}\frac{h^{n+1}}{\varphi}d\sigma}\right)^{\frac{1}{n+1}}, & \hbox{$p=n+1$.}
  \end{array}
\right.
\]
Therefore, $\lambda K$ solves the smooth, even $L_p$-Minkowski problem.
\end{proof}

Let us close this section with a brief outline of this paper. The main difficulty in proving convergence of the normalized solutions is in obtaining long-time existence. The issue arises from the time-dependent anisotropic factor (the support function). We believe in such generality, (\ref{e: flow0}) serves as the first example where a time-dependent anisotropic factor is allowed. To prove long-time existence, we first obtain bounds on the Gauss curvature in Section \ref{subsec:gauss-curvature-bounds}. Using the well-known standard technique of Tso \cite{Tso} we obtain upper bounds. We obtain lower bounds by applying the same technique to the evolution of the polar body as in \cite{Ivaki-Proc}. Controlling the principal curvatures requires estimates of higher derivatives of the speed which is generally quite difficult due to the non-linearity of the flow. In Section \ref{subsec:principal-curvature-bounds} we obtain these crucial estimates by adapting the remarkable $C^2$ estimates of Guan-Ren-Wang for the prescribed curvature problem see \cite[(4.2)]{Guan}. Long time existence then follows readily by standard arguments. Once it is proved that solutions to the flow exist until they expand to infinity uniformly in all directions, the method of \cite[Section 8]{Ivaki2014-gauss} applies and yields convergence of the volume-normalized solutions in $C^{1}$ to self-similar solutions provided $p\neq1$. Further work is required to establish convergence of normalized solutions if $p=1$, and to prove convergence in $C^{\infty}$ for $p\leq n+1$; this is accomplished in Section \ref{sec: nor conv}; see also Remark \ref{rem}.
\section*{Acknowledgment}
The work of the first author was supported by the EPSRC on a Programme Grant entitled ``Singularities of Geometric Partial Differential Equations'' reference number EP/K00865X/1. The work of the second author was supported by Austrian Science Fund (FWF) Project
M1716-N25 and the European Research Council (ERC) Project 306445.


\section{basic evolution equations}

Let $g=\{g_{ij}\}$, and $W=\{w_{ij}\}$ denote, in order, the induced metric and the second fundamental form of $M$. At every point in the hypersurface $M$ choose a local orthonormal frame $\{e_1,\ldots, e_n\}.$

We use the following standard notation
\[w_i^j=g^{mj}w_{im},\]
\[(w^2)_i^j=g^{mj}g^{rs}w_{ir}w_{sm},\]
\[|W|^2=g^{ij}g^{kl}w_{ik}w_{lj}=w_{ij}w^{ij}.\]
Here, $\{g^{ij}\}$ is the inverse matrix of $\{g_{ij}\}.$

We use semicolons to denote covariant derivatives. The following geometric formulas are well-known:
\begin{align*}
\nu_{;i} &= w_i^ke_k,\\
\nu_{;ij} &= g^{kl}w_{ij;l}e_k - w_i^lw_{lj}\nu,\\
h_{;i} &= w_i^k (F\cdot e_k),\\
h_{;ij} &= w_{ij} - hw_i^lw_{lj} + F \cdot \nabla w_{ij}.\\
\end{align*}
Note that in above we considered the support function as a function on the boundary of the hypersurface; that is, at the point $x\in M$ we have
$$h(x)=F(x)\cdot \nu(x).$$

For convenience, let $\psi(x)=h^{2-p}(x)\varphi(\nu(x))$. The following evolution can be deduced in a standard manner; see for example \cite{Gerhardt:/2006}.
\begin{lemma} The following evolution equations hold:
\[\partial_t \nu = -\nabla \left(\frac{\psi}{\mathcal{K}}\right),\]
\begin{align*}
\partial_t w_i^j &= -\left(\frac{\psi}{\mathcal{K}}\right)_{;ik}g^{kj} - \left(\frac{\psi}{\mathcal{K}}\right) w_i^kw_k^j \\
&= \psi \frac{\mathcal{K}^{kl}}{\mathcal{K}^2} w_{i;kl}^j + \psi \frac{\mathcal{K}^{kl}}{\mathcal{K}^2} w_{kr}w_l^rw_i^j -(n+1)\frac{\psi }{\mathcal{K}} w_{i}^k w_k^{j} \\
&\quad  + \psi\frac{\mathcal{K}^{kl,rs}}{\mathcal{K}^{2}}g^{jm} w_{kl;i}w_{rs;m}- \frac{2\psi}{\mathcal{K}^{3}}g^{jm} \mathcal{K}_{;i}\mathcal{K}_{;m} \\
&\quad + \frac{1}{\mathcal{K}^{2}}g^{jk}\mathcal{K}_{;k}\psi_{;i} + \frac{1}{\mathcal{K}^{2}} g^{jk}\psi_{;k}\mathcal{K}_{;i} - \frac{1}{\mathcal{K}} g^{jk}\psi_{;ik},
\end{align*}
\begin{align*}
\partial_th &= \psi \frac{\mathcal{K}^{ij}}{\mathcal{K}^{2}} h_{;ij} + \psi h \frac{\mathcal{K}^{ij}}{\mathcal{K}^{2}} w_i^lw_{lj} - (n-1) \frac{\psi}{\mathcal{K}}- \frac{1}{\mathcal{K}} F\cdot\nabla\psi.
\end{align*}
\end{lemma}
\section{long-time existence}
\label{sec:long-time-existence}
\subsection{Lower and upper bounds on Gauss curvature}
\label{subsec:gauss-curvature-bounds}
The proofs of the following two lemmas are similar to the proofs of \cite[ Lemmas 4.1, 4.2]{Ivaki-Proc}. For completeness, we give the proofs here. In this section we use $\bar{\nabla}$ to denote covariant derivatives on the sphere with respect to the standard metric.

The matrix of the radii of the curvature of a smooth, closed, strictly convex hypersurface is denoted by $\mathfrak{r}=[\mathfrak{r}_{ij}]$ and the entries of $\mathfrak{r}$ are considered as functions on the unit sphere. They can be expressed in terms of the support function as $\mathfrak{r}_{ij}:=\bar{\nabla}^{2}_{ij} h+\bar{g}_{ij}h,$ where $[\bar{g}_{ij}]$ is the standard metric on $\mathbb{S}^{n}$. Additionally, we recall that
$S_n=\det [\mathfrak{r}_{ij}]/\det[\bar{g}_{ij}].$
\begin{lemma}\label{lem: lower}
Let $\{M_t\}$ be a solution of (\ref{e: flow0}) on $[0,t_1]$. If $c_2\leq h_{K_t}\leq c_1$ on $[0,t_1]$, then $\mathcal{K}\leq c_4$ on $[0,t_1].$ Here $c_4$ depends on $K_0$, $c_1,c_2,p,\varphi$ and $t_1.$
\end{lemma}
\begin{proof}
We apply the maximum principle to the following auxiliary function defined on the unit sphere
\[\Theta=\frac{\psi S_n}{2c_1-h}=\frac{\partial_t h}{2c_1-h}.\]
At any minimum of $\Theta$ we have
\[0=\bar{\nabla}_i\Theta=\bar{\nabla}_i \left(\frac{\psi S_{n}}{2c_1-h}\right)\ \ \ {\hbox{and}}\ \ \  \bar{\nabla}^{2}_{ij}\Theta\geq 0.\]
Therefore, we get
\[ \frac{\bar{\nabla}_i (\psi S_{n})}{2c_1-h}=-\frac{\psi S_{n} \bar{\nabla}_i h}{(2c_1-h)^2} \] and
\begin{equation}\label{e: tso dual}
\bar{\nabla}^{2}_{ij} (\psi S_{n})+\bar{g}_{ij} \psi S_{n}\geq
\frac{-\psi S_{n}\mathfrak{r}_{ij}+2c_1\psi S_{n}\bar{g}_{ij}}{2c_1-h}.
\end{equation}
Differentiating $\Theta$ with respect to time yields
\begin{align*}
\partial_t\Theta&=\frac{\psi S_n^{ij}}{2c_1-h}
\left(\bar{\nabla}^{2}_{ij}(\psi S_{n})+\bar{g}_{ij}\psi S_{n}\right)
+\frac{\psi^2S_n^{2}}{(2c_1-h)^2}\left(1+(2-p)h^{-1}(2c_1-h)\right),
\end{align*}
 where $S_n^{ij}$ is the derivative of $S_n$ with respect to the entry $ \mathfrak{r}_{ij}$.
By applying inequality (\ref{e: tso dual}) to the preceding identity we deduce
\begin{equation}\label{e: last step tso dual}
\partial_t\Theta\geq \Theta^2\left(1-n+2 c_1\mathcal{H}\right)-c\Theta^2,
\end{equation}
where
\[\mathcal{H}=S_{n}^{-1}S_{n}^{ij}\bar{g}_{ij}.\]
Therefore, we arrive at
\[\frac{\varphi\frac{h^{2-p}}{\mathcal{K}}}{2c_1-h}(t, u)\geq \frac{1}{ct+1/\min\limits_{u\in\mathbb{S}^{n}}\frac{\varphi\frac{h^{2-p}}{\mathcal{K}}}{2c_1-h}(0, u)}\geq \frac{1}{ct_1+1/\min\limits_{u\in\mathbb{S}^{n}}\frac{\varphi\frac{h^{2-p}}{\mathcal{K}}}{2c_1-h}(0, u)}.\]
\end{proof}

\begin{lemma}\label{lem: upper}
Let $\{M_t\}$ be a solution of (\ref{e: flow0}) on $[0,t_1]$. If $c_1\leq h_{K_t}\leq c_2$ on $[0,t_1]$, then $\mathcal{K}\geq \frac{1}{a+b t^{-\frac{n}{n+1}}}$ on $(0,t_1],$ where $a$ and $b$ depend only on $c_1,c_2,p,\varphi.$ In particular, $\mathcal{K}\ge c_3$ on $[0,t_1]$ for a positive number $c_{3}$ that depends on $K_0$, $c_1,c_2,p,\varphi$ and is independent of $t_1.$
\end{lemma}
\begin{proof}
Suppose $K_t^{\ast}$ is the polar body\footnote{The polar body of a convex body $K$ with the origin of $\mathbb{R}^{n+1}$ in its interior is the convex body defined by
\[K^{\ast}=\{x\in\mathbb{R}^{n+1}: x\cdot y\leq 1 \mbox{~for~all~}y\in K\}.\]} of $K_t$ with respect to the origin. We furnish quantities associated with polar bodies with $^\ast$. The polar bodies evolve by
$$\partial_th^{\ast}=-\psi^{\ast} S_{n}^{\ast-1},\quad h^{\ast}(\cdot,t)=h_{K_{t}^{\ast}}(\cdot),$$
where $$\psi^{\ast}=\frac{(h^{\ast2}+|\bar{\nabla}h^{\ast}|^2)^{\frac{n+1+p}{2}}}{h^{\ast n+1}}\varphi\left(\frac{h^{\ast}u+\bar{\nabla} h^{\ast}}{\sqrt{h^{\ast2}+|\bar{\nabla}h^{\ast}|^2}}\right);$$
see Lemma \ref{app1} for the proof.
In addition, we have $c_1'=1/c_2\leq h^{\ast}\leq 1/c_1=c_2'.$ We will show that the function
\[\Theta=\frac{\psi^{\ast} S_{n}^{\ast-1}}{h^{\ast}-c_1'/2}\]
remains bounded.
At any maximal point of $\Theta:$
\[0=\bar{\nabla}_i\Theta=\bar{\nabla}_i \left(\frac{\psi^{\ast} S_{n}^{\ast-1}}{h^{\ast}-c_1'/2}\right)\ \ \ {\hbox{and}}\ \ \  \bar{\nabla}^{2}_{ij} \Theta\leq 0.\]
Hence, we obtain
\begin{equation}\label{e: grad of Q}
\frac{\bar{\nabla}_i (\psi^{\ast} S_{n}^{\ast-1})}{h^{\ast}-c_1'/2}=\frac{\psi^{\ast} S_{n}^{\ast-1} \bar{\nabla}_i h^{\ast}}{(h^{\ast}-c_1'/2)^2},
\end{equation}
and consequently,
\begin{equation}\label{e: tso}
\bar{\nabla}^{2}_{ij}(\psi^{\ast} S_{n}^{\ast-1})+\bar{g}_{ij}\psi^{\ast} S_{n}^{\ast-1}\leq
\frac{\psi^{\ast} S_{n}^{\ast-1}\mathfrak{r}_{ij}^{\ast}-\frac{c_1'}{2}\psi^{\ast} S_{n}^{\ast-1}
\bar{g}_{ij}}{h^{\ast}-c_1'/2}.
\end{equation}
Differentiating $\Theta$ with respect to time yields
\begin{align*}
\partial_t\Theta=&\frac{\psi^{\ast}S_n^{\ast -2}}{h^{\ast}-c_1'/2}S_n^{\ast ij}
\left(\bar{\nabla}^{2}_{ij}(\psi^{\ast} S_{n}^{\ast-1})+\bar{g}_{ij}
\psi^{\ast} S_{n}^{\ast-1}\right)\\
&+\frac{S_{n}^{\ast-1}}{h^{\ast}-c_1'/2}\partial_t \psi^{\ast}+\Theta^2.
\end{align*}
On the other hand, in view of
\[|\partial_t h^{\ast}|=\psi^{\ast} S_{n}^{\ast-1},\quad \|\bar{\nabla}\partial_t h^{\ast}\|=\|\bar{\nabla}(\psi^{\ast} S_{n}^{\ast-1})\|=\frac{\psi^{\ast} S_{n}^{\ast-1} \|\bar{\nabla} h^{\ast}\|}{h^{\ast}-c_1'/2},\quad\|\bar{\nabla} h^{\ast}\|\leq c_2',\]
where for the second equation we used (\ref{e: grad of Q}),  we have
\[\frac{S_{n}^{\ast-1}}{h^{\ast}-c_1'/2}\partial_t \psi^{\ast}\leq c(n,p,c_1,c_2,\varphi)\Theta^2.\]
Employing this last inequality and inequality (\ref{e: tso}) we infer that, at any point where the maximum of $\Theta$ is reached, we have
\begin{equation}\label{e: last step tso}
\partial_t\Theta\leq\Theta^2\left(c'-\frac{c_1'}{2}\mathcal{H}^{\ast}\right).
\end{equation}
Moreover, we have
\begin{align*}
\mathcal{H}^{\ast}&\geq n\left(\frac{h^{\ast}-c_1'/2}{\psi^{\ast} S_{n}^{\ast-1}}\right)^{-\frac{1}{n}}
\left(\frac{\psi^{\ast}}{h^{\ast}-c_1'/2}\right)^{-\frac{1}{n}}\\
&\geq n\Theta^{\frac{1}{n}} \left(\frac{c''}{c_1'-c_1'/2}\right)^{-\frac{1}{n}}.
\end{align*}
Therefore, we can rewrite the inequality (\ref{e: last step tso}) as follows
\begin{align*}
\partial_t\Theta&\leq \Theta^2\left(c-c'\Theta^{\frac{1}{n}} \right),
\end{align*} for positive constants $c$ and $c'$ depending only on $p,c_1,c_2,\varphi.$
Hence,
\begin{equation}\label{ie: upper on psi}
\Theta\leq c+c't^{-\frac{n}{n+1}}
\end{equation}
for some positive constants depending only on $p,c_1,c_2,\varphi.$
\footnote{
\begin{claim}Suppose $f$ is a positive smooth function of $t$ on $[0,t_1]$ that satisfies
\begin{align}\label{claim}
\frac{d}{dt}f\leq c_0+c_1f+c_2f^2-c_3f^{2+p},
\end{align}
where $c_3,p$ are positive. There exist constant $c,c'>0$ independent of the solution and depending only on $c_0,c_1,c_2,c_3,p$, such that $f\leq c+c't^{-1/(p+1)}~\mbox{on}~(0,t_1].$
\end{claim}
\begin{proof}
Note that there exists $x_0>0$ such that $c_0+c_1x+c_2x^2-c_3x^{2+p}<-c_3/2x^{2+p}$ for $x>x_0.$ If $f(0)\leq x_0$, then $f$ may increase forward in time, but when $f$ reaches $x_0$, then $f$ must start decreasing (since the right-hand side of (\ref{claim}) becomes negative).
Thus we may assume, without loss of generality, that $f(0)>x_0.$ Therefore, $f>x_0$ on a maximal time interval $[0,t_0).$ On $[0,t_0)$ we can solve
$$\frac{d}{dt}f\le-c_3/2f^{2+p}$$ to obtain
$$f\leq (c_3(p+1)/2t)^{-1/(p+1)}.$$
At $t_0$ we have $c_0+c_1f+c_2f^2-c_3f^{1+p}=-c_3/2f^{2+p}$ and $f=x_0;$ therefore the right-hand side of (\ref{claim}) is still negative. So $f\le f(t_0)$ on $[t_0,t_1].$ In conclusion, $$f\leq \max\{(c_3(p+1)/2t)^{-1/(p+1)}, x_0=f(t_0)\}\leq c+c't^{-1/(1+p)},$$ where $c,c'$ do not depend on solutions.
\end{proof}}
The inequality (\ref{ie: upper on psi}) implies that
 \begin{equation}\label{ie: upper on psi1}
S_n^{\ast -1}\leq a'+b' t^{-\frac{n}{n+1}}\end{equation}
for some $a'$ and $b'$ depending only on $p,c_1,c_2,\varphi.$ Now we can use the argument given in \cite[Lemma 2.3]{Ivaki-Stancu} to obtain the desired lower bound:
For every $u \in \mathbb{S}^{n}$, there exists a unique $u^\ast \in \mathbb{S}^{n}$ such that
$$\left(S_nh^{n+2}\right)(u)\left(S_n^{\ast}h^{\ast n+2}\right)(u^{\ast})=1,$$
see \cite{Hug}.
In view of this identity and (\ref{ie: upper on psi1}) we conclude that on $(0,t_1]$ we have
$$\mathcal{K}\geq \frac{1}{a+b t^{-\frac{n}{n+1}}}$$ for some $a$ and $b$ depending only on $p,c_1,c_2,\varphi.$ The lower bound for $\mathcal{K}$ on $[0,\delta]$ for a small enough $\delta>0$ follows from the short-time existence of the flow. The lower bound for $\mathcal{K}$ on $[\delta,t_1]$ follows from the inequality $\mathcal{K}\geq \frac{1}{a+b \delta^{-\frac{n}{n+1}}}.$
\end{proof}

\subsection{Upper and lower bounds on principal curvatures}
\label{subsec:principal-curvature-bounds}
 To obtain upper and lower bounds on the principal curvatures, denoted by $\{\kappa_i\}_{i=1}^n$, we will consider the auxiliary function used by Guan-Ren-Wang for a prescribed curvature problem; see \cite[(4.2)]{Guan}.
\begin{lemma}\label{lem: final}
Let $\{M_t\}$ be a solution of (\ref{e: flow0}) on $[0,t_1]$. If $c_1\leq h_{K_t}\leq c_2$ on $[0,t_1]$, then $c_5\leq \kappa_i\le c_6$ on $[0,t_1],$ where $c_5$ and $c_6$ depend on $K_0$, $c_1,c_2,p,\varphi$ and $t_1.$
\end{lemma}
\begin{proof}
In view of Lemmas \ref{lem: lower} and \ref{lem: upper}, it suffices to show that $\|W\|$ remains bounded on $[0,t_1]$. Consider the auxiliary function
\[\Theta=\frac 12\log(\|W\|^2)-\alpha\log h.\]
Assume without loss of generality that $c_1>1,$ for otherwise we replace $h$ by $2h/c_{1}$, which does not effect the evolution equation of $\Theta$.
Using the parabolic maximum principle we show that for some $\alpha$ large enough $\Theta(\cdot,t)$ is always negative on $[0,t_1]$. If the conclusion of the theorem is false, we may choose $(x_0,t_0)$ with $t_0>0$ and such that $\Theta(x_0,t_0)=0$, $\Theta(x, t_0) \leq 0$, and $\Theta(x,t) < 0$ for $t<t_0$. Then,
\begin{align*}
0 &\leq \dot{\Theta}-\psi\frac{\mathcal{K}^{kl}}{\mathcal{K}^2}\Theta_{;kl} \\
&= -\frac{\psi}{\|W\|^2} \frac{\mathcal{K}^{kl}}{\mathcal{K}^2} w_{i;k}^j w_{j;l}^i + \frac{2\psi}{\|W\|^4} \frac{\mathcal{K}^{kl}}{\mathcal{K}^2} w^j_iw^s_r w^i_{j;k} w^r_{s;l} \\
&+ \psi \frac{\mathcal{K}^{kl}}{\mathcal{K}^2} w_{kr}w_l^r - (n+1)\psi \frac{(w^2)_i^jw_j^i}{\mathcal{K}\|W\|^2} \\
&+ \frac{\psi w^i_j}{\|W\|^2} \left(\frac{\mathcal{K}^{kl,rs}}{\mathcal{K}^2} w_{kl;i} g^{jp}w_{rs;p} - 2\frac{g^{jp} \mathcal{K}_{;i}\mathcal{K}_{;p}}{\mathcal{K}^3}\right) \\
&+ \left(\frac{2}{\mathcal{K}^2} g^{jp}\psi_{;i} \mathcal{K}_{;p} - \frac{1}{\mathcal{K}} g^{jp} \psi_{;ip}\right) \frac{w^i_j}{\|W\|^2} \\
& +(n-1) \frac{\alpha\psi}{h\mathcal{K}} + \frac{\alpha}{h\mathcal{K}} (F\cdot\nabla\psi) - \frac{\alpha\psi}{h^2} \frac{\mathcal{K}^{kl}}{\mathcal{K}^2} h_{;k}h_{;l} - \alpha\psi \frac{\mathcal{K}^{kl}}{\mathcal{K}^2} w_{kr}w_l^r.
\end{align*}
Pick normal coordinates around $x_0$ such that in $(x_0,t_0)$ there holds
\[g_{ij}=\delta_{ij}, w_{ij}=w_{ii}\delta_{ij}.\]
At $(x_0,t_0)$ we may write $$\mathcal{K}^{kl,rs}w_{kl;i}w_{rs;i}=\mathcal{K}^{kk,ll}w_{kk;i}w_{ll;i}-\mathcal{K}^{kk,ll}w_{kl;i}^2,$$ due to the relation
\begin{align}\begin{split}\label{eq: final-1}
\mathcal{K}^{kl,rs}w_{kl;i}w_{rs;j}w^{ij}=\sum_i w_{ii}\Big(&\sum_{p,q}\frac{\partial^2\mathcal{K}}{\partial \kappa_p\partial\kappa_q}w_{pp;i}w_{qq;i}\\
		&+\sum_{p\neq q}\frac{\frac{\partial\mathcal{K}}{\partial \kappa_p}-\frac{\partial\mathcal{K}}{\partial\kappa_q}}{\kappa_p-\kappa_q} w_{pq;i}^2\Big),
\end{split}\end{align}
see for example \cite[Lemma~2.1.14]{Gerhardt:/2006}.
We obtain after multiplication by $\mathcal{K}^{2}$ that
\begin{align*}
0 \leq& -\frac{\psi}{\|W\|^2} \mathcal{K}^{ii} \sum_lw_{ll;i}^2 - \frac{\psi}{\|W\|^2} \mathcal{K}^{ii} \sum_{p\ne q}w_{pq;i}^2 + \frac{2\psi}{\|W\|^4} \mathcal{K}^{ii} \left(\sum_jw_{jj}w_{jj;i}\right)^2 \\
&+ \psi\mathcal{K}^{ii} w_{ii}^2 -(n+1) \psi \mathcal{K} \sum_i\frac{w_{ii}^3}{\|W\|^2} \\
&+ \frac{\psi}{\|W\|^2} \sum_{i}w_{ii}\left(\mathcal{K}^{pp,qq}w_{pp;i}w_{qq;i} - \mathcal{K}^{pp,qq}w_{pq;i}^2 - 2\frac{(\mathcal{K}_{;i})^2}{\mathcal{K}}\right) \\
&+ \sum_{i}(2\psi_{;i}\mathcal{K}_{;i} - \mathcal{K}\psi_{;ii}) \frac{w_{ii}}{\|W\|^2} \\
&+ (n-1) \frac{\alpha\psi\mathcal{K}}{h} + \frac{\alpha\mathcal{K}}{h}(F\cdot\nabla\psi) - \frac{\alpha\psi}{h^2}\mathcal{K}^{kl} h_{;k} h_{;l} - \alpha\psi\mathcal{K}^{ii}w_{ii}^2.
\end{align*}
At $(x_0,t_0)$ we have
\begin{equation}\label{eq: max}
0 = \Theta_{;k} = \sum_{i}\frac{w_{ii}w_{ii;k}}{\|W\|^2} - \alpha\frac{h_{;k}}{h},
\end{equation}
We may assume at $x_0$ that $w_{11}=\max\{w_{ii}:1\leq i\leq n\}.$ Therefore,
\begin{equation}\label{eq: max1}
\Theta(x_0,t_0)=0\Rightarrow\frac{c_1^{\alpha}}{\sqrt{n}}\leq w_{11}\leq c_2^{\alpha}.
\end{equation}
On the other hand, since $\psi$ is bounded above and below in view of the hypotheses of the lemma, we obtain
\begin{align}\label{eq: 1}
\psi_{;i} \leq C_0w_{ii} \Rightarrow 2\psi_{;i} \mathcal{K}_{;i}&\leq \frac{\varepsilon\psi}{c_4} (\mathcal{K}_{;i})^2 + \frac{c_4C_0^2}{\psi\varepsilon}w_{ii}^2\nonumber\\
&\leq \varepsilon\psi \frac{(\mathcal{K}_{;i})^2}{\mathcal{K}} + C(\varepsilon,K_0,\varphi,t_1) \psi w_{ii}^{2},
\end{align}
where $c_4$ (depending on $t_1$) is from Lemma \ref{lem: lower},
and
\begin{equation}\label{eq: 2}
\psi_{;ii} \geq - C - Cw_{ii} - Cw_{ii}^2 + \sum_{k}w_{ii;k} d_{\nu} \psi(\partial_k).
\end{equation}
Using (\ref{eq: max}) in (\ref{eq: 2}) we obtain
\begin{align}\begin{split}\label{eq: 3}
&-\frac{\mathcal{K}}{\|W\|^2}\sum_iw_{ii} \psi_{;ii}\\
	\leq &\frac{\mathcal{K}}{\|W\|^2}\sum_i w_{ii}(C+Cw_{ii}+Cw_{ii}^2-\sum_{k}w_{ii;k}d_{\nu}\psi(\partial_k))\\
\leq &\frac{\mathcal{K}}{\|W\|^2}\sum_i w_{ii}(C+Cw_{ii}+Cw_{ii}^2)-\frac{\alpha\mathcal{K}}{h} \sum_k h_{;k} d_{\nu}\psi(\partial_k)\\
=&\frac{\mathcal{K}}{\|W\|^2}\sum_i w_{ii}(C+Cw_{ii}+Cw_{ii}^2)-\frac{\alpha\mathcal{K}}{h}\sum_i w_{ii}(\partial_i\cdot F) d_{\nu}\psi(\partial_i)\\
\leq &\frac{\psi}{\|W\|^2}\sum_i w_{ii}(C+Cw_{ii}^2)-\frac{\alpha\mathcal{K}}{h}\sum_i w_{ii} (\partial_i \cdot F) d_{\nu}\psi(\partial_i).
\end{split}\end{align}
For the last inequality, we used that $\mathcal{K}$ is bounded above and $\psi$ is bounded below (so the constant $C$ depends on $K_0,\varphi, t_1$).

Combining (\ref{eq: max}), (\ref{eq: 1}) and (\ref{eq: 3}) implies that
\begin{align}\begin{split}\label{eq:L10-1}
0\leq& -\frac{\psi}{\|W\|^2}\mathcal{K}^{ii}\sum_lw_{ll;i}^2-\frac{\psi}{\|W\|^2}\mathcal{K}^{ii}\sum_{p\ne q}w_{pq;i}^2\\
&+\frac{2\psi}{\|W\|^4}\mathcal{K}^{ii}\left(\sum_jw_{jj}w_{jj;i}\right)^2+\psi\mathcal{K}^{ii}w_{ii}^2-(n+1)\psi\mathcal{K}\sum_i\frac{w_{ii}^3}{\|W\|^2}\\
&+\frac{\psi}{\|W\|^2}\sum_{l}w_{ll}\left(\mathcal{K}^{pp,qq}w_{pp;l}w_{qq;l}-\mathcal{K}^{pp,qq}w_{pq;l}^2
-(2-\varepsilon)\frac{(\mathcal{K}_{;l})^2}{\mathcal{K}}\right)\\
&+\frac{\psi}{\|W\|^2}\sum_i w_{ii}(C+Cw_{ii}^2)-\frac{\alpha\mathcal{K}}{h}\sum_i w_{ii}(\partial_i\cdot F) d_{\nu}\psi(\partial_i)\\
&+(n-1)\frac{\alpha\psi\mathcal{K}}{h}+\frac{\alpha\mathcal{K}}{h}\sum_s(\partial_s\cdot F) d_F\psi(\partial_s)+\frac{\alpha\mathcal{K}}{h}\sum_iw_{ii}(\partial_i\cdot F) d_{\nu}\psi(\partial_i)\\
&-\frac{\alpha\psi}{h^2}\mathcal{K}^{ii}w_{ii}^2(\partial_i\cdot F)^2 -\alpha\psi\mathcal{K}^{ii}w_{ii}^2\\
\leq&\frac{\psi}{\|W\|^2}\left(\sum_lw_{ll}\left(C+Cw_{ll}^2\right)-n\mathcal{K}\sum_lw_{ll}^3+\mathcal{K}^{ii}w_{ii}^2\|W\|^2\right)\\
&+\alpha\psi\left(\frac{n\mathcal{K}}{h}-\mathcal{K}^{ii}w_{ii}^2-\frac{\mathcal{K}^{ii}w_{ii}^2(\partial_i\cdot F)^2}{h^2}+\frac{\mathcal{K}}{h\psi}\sum_s(\partial_s\cdot F) d_F\psi(\partial_s)\right)\\
&-\psi\sum_i\left(A_i+B_i+C_i+D_i-E_i\right)
-\frac{\alpha\psi\mathcal{K}}{h}-\psi\mathcal{K}\sum_i\frac{w_{ii}^3}{\|W\|^2},
\end{split}\end{align}
where $C$ depends on $\varepsilon,K_0,\varphi,t_1$, and
\[A_i=\frac{2-\varepsilon}{\|W\|^2\mathcal{K}} w_{ii}(\mathcal{K}_{;i})^2-\frac{w_{ii}}{\|W\|^2}\sum_{p,q}\mathcal{K}^{pp,qq}w_{pp;i}w_{qq;i},\]

\[B_i=\frac{2}{\|W\|^2}\sum_{j}w_{jj}\mathcal{K}^{jj,ii}w^2_{jj;i},\quad C_i=\frac{2}{\|W\|^2}\sum_{j\neq i}\mathcal{K}^{jj}w_{jj;i}^2,\]

\[D_i=\frac{1}{\|W\|^2}\mathcal{K}^{ii}\sum_j w_{jj;i}^2,\quad
E_i=\frac{2}{\|W\|^4}\mathcal{K}^{ii}\left(\sum_j w_{jj}w_{jj;i}\right)^2.\]

The terms $B_i$ and $C_i$ deserve some explanation. $C_i$ comes from the second term in \eqref{eq:L10-1}, which reads
\[-\frac{\psi}{\|W\|^2}\sum_i\mathcal{K}^{ii}\sum_{p\neq q}w_{pq;i}^2\leq -\frac{\psi}{\|W\|^2}\sum_{p\neq q}\mathcal{K}^{pp}w_{pq;p}^2-\frac{\psi}{\|W\|^2}\sum_{p\neq q}\mathcal{K}^{qq}w_{pq;q}^2, \]
which is exactly $C_i$ due to the Codazzi equation.

The third line of \eqref{eq:L10-1} arises from \eqref{eq: final-1}.
 Since the second term in the bracket of \eqref{eq: final-1} is negative and the hypersurface is convex, we can proceed in the same way as we derived $C_i$ and just throw away all indices $i$ which are neither $p$ nor $q$. This gives term $B_i$. The first term in the big bracket goes into $A_i$.

In Corollary \ref{Alternative} of the appendix we will present an adaption of the method developed in \cite{Guan} to deal with the curvature derivative terms $A_i,B_i,C_i,D_i,E_i$. There we prove that we obtain the following alternative: There exist positive numbers $\delta_2,\dots,\delta_n$ which only depend on the dimension and bounds on the Gauss curvature, such that either
\[w_{ii}> \delta_iw_{11}\quad\forall 2\leq i\leq n\]
or
\[A_i+B_i+C_i+D_i-E_i\geq 0\quad\forall 1\leq i\leq n.\]
By taking $\alpha$ large in \eqref{eq: max1}, in the first case we get a contradiction to the bound on the Gauss curvature. In the second case, using also $\mathcal{K}^{ii}w_{ii}^2 = \mathcal{K} \sum_i w_{ii}$, \eqref{eq:L10-1} yields

\begin{align*}
0 \leq& \frac{\psi}{\|W\|^2}\left(\sum_l w_{ll}(C + Cw_{ll}^2) - n \mathcal{K} \sum_l w_{ll}^3\right)- (\alpha-1) \mathcal{K} \psi \sum_i w_{ii} \\
& + \alpha\psi \left((n-1)\frac{\mathcal{K}}{h} - \frac{\mathcal{K}}{h^2} \sum_{i} w_{ii} (\partial_i \cdot F)^2 + \frac{\mathcal{K}}{h\psi} \sum_l  (\partial_l \cdot F) d_F \psi (\partial_l) \right).
\end{align*}
Consequently we obtain
\[0\leq \frac{C(\varepsilon,K_0,\varphi,t_1)w_{11}^3}{\|W\|^2} - (\alpha-1) \mathcal{K}\psi w_{11} + C(K_0,\varphi,t_1) \alpha,\]
where we discarded $-(\alpha-1)\mathcal{K}\psi \sum_{i\neq 1} w_{ii} \leq 0$ and used the bounds on $h, \psi$ and $\mathcal{K}$ to bound $w_{11}$ in terms of $w_{11}^3$.

Now take $\alpha$ such that $(\alpha-1) \mathcal{K}\psi \geq C(\varepsilon,K_0,\varphi,t_1)+1$. Therefore, in view of (\ref{eq: max1})
\begin{equation}\label{x}
\begin{split}
0 &\leq \frac{C(\varepsilon,K_0,\varphi,t_1)w_{11}^3}{\|W\|^2} - (\alpha-1) \mathcal{K}\psi w_{11} + C(K_0,\varphi,t_1)\alpha \\
&\leq C(\varepsilon,K_0,\varphi,t_1)\left(\frac{w_{11}^2}{\|W\|^2} - 1\right) w_{11} - w_{11} + C(K_0,\varphi,t_1)\alpha \\
&\leq -\frac{c_1^{\alpha}}{\sqrt{n}} + C(K_0,\varphi,t_1)\alpha.
\end{split}
\end{equation}
Taking $\alpha$ large enough yields a contradiction.
\end{proof}
\begin{proposition}\label{prop: expansion to infty}
The solution to (\ref{e: flow0}) satisfies $\lim\limits_{t\to T}\max h_{K_t}=\infty.$
\end{proposition}
\begin{proof}
First, let $p\geq n+1.$ In this case, by comparing with suitable outer balls, the flow exists on $[0,\infty).$
For $p>n+1,$ consider an origin centered ball $B_r$, such that $K_0\supseteq B_r.$ Then $K_t\supseteq B_{r(t)},$ where $$r(t)=\left((\min h_{K_0})^{p-n-1}+t(p-n-1)\min \varphi \right)^{\frac{1}{p-n-1}}$$ and $B_{r(t)}$ expands to infinity as $t$ approaches $\infty$. For $p=n+1$, $K_t\supseteq B_{r(t)}$ with $r(t)=e^{t\min\varphi}\min h_{K_0}$ and $B_{r(t)}$ expands to infinity as $t$ approaches $\infty$.

Second, if $p<n+1$, then the flow exists only on a finite time interval. If $\max h_{K_t}<\infty$, then by Lemmas \ref{lem: lower}, \ref{lem: upper} and \ref{lem: final}, the evolution equation (\ref{e: flow0}) is uniformly parabolic on $[0,T)$. Thus, the result of Krylov and Safonov \cite{Krylov-Safonov} and standard parabolic theory allow us to extend the solution smoothly past time $T$, contradicting its maximality.
\end{proof}

\section{convergence of normalized solutions}\label{sec: nor conv}
\subsection{Convergence in \texorpdfstring{$C^{1}, 1\neq p>-n-1$}{}}\label{sec}
By the proof of \cite[Corollary 7.5]{Ivaki2014-gauss}, there exist $r,R$ such that
\begin{align}\label{ratio}
0<r\leq h_{\tilde{K}_t}\leq R<\infty.
\end{align}
Therefore, a subsequence of $\{\tilde{K}_{t_k}\}$ converges in the Hausdorff distance to a limiting shape $\tilde{K}_{\infty}$ with the origin in its interior. The argument of \cite[Section 8.1]{Ivaki2014-gauss} implies  $$\varphi h_{\tilde{K}_{\infty}}^{1-p}f_{\tilde{K}_{\infty}}=c,$$ where $f_{\tilde{K}_{\infty}}$ is the positive continuous curvature function of $\tilde{K}_{\infty}$ and $c$ is some positive constant. By \cite[Proposition 1.2]{39}, $\tilde{K}_{\infty}$  is smooth and strictly convex. The $C^1$-convergence follows, which is purely geometric and does not depend on the evolution equation, from \cite[Lemma 13]{Andrews1997}.
\begin{remark}\label{rem}
Section \ref{sec} completes the discussion on the existence of solutions to the smooth, even $L_p$-Minkowski problems in $\mathbb{R}^{n+1}$ for $1\neq p>-n-1.$ The next section discusses  the $C^{\infty}$ convergence when $1\neq p\leq n+1$, and also when $p=1$ and solutions are origin-symmetric. We mention that in the latter case, by the proof of \cite[Corollary 7.5]{Ivaki2014-gauss}, the estimate (\ref{ratio}) still holds.
\end{remark}

\subsection{Convergence in \texorpdfstring{$C^{\infty}$}{}}\label{convergence-1}

By \cite[Lemma 9.2]{Ivaki2014-gauss}, there is a uniform upper bound on the Gauss curvature of the normalized solution when $p\leq n+1.$
In the following, we first obtain a uniform lower bound on the Gauss curvature of the normalized solution $\tilde{K}_t$.

Let $h:\mathbb{S}^{n}\times[0,T)\to\mathbb{R}^{n+1}$ be a solution of equation (\ref{eq: flow4}). Then for each $\lambda >0$, $\bar{h}$ defined by
\begin{align*}
\bar{h}&:\mathbb{S}^{n}\times\left[0,T/\lambda^{\frac{1+n-p}{n+1}}\right)\to\mathbb{R}^{n+1}\\
\bar{h}(u,t)&=\lambda^{\frac{1}{n+1}} h\left(u,\lambda^{\frac{1+n-p}{n+1}}t\right)
\end{align*}
is also a solution of evolution equation (\ref{eq: flow4}) but with the initial data $\lambda^{\frac{1}{n+1}} h\left(\cdot,0\right).$

For each \emph{fixed} time $t\in[0,T),$ define $\bar{h}$ a solution of (\ref{eq: flow4}) as follows
\begin{align*}
\bar{h}(u,\tau)=\left(\frac{V(B)}{V(K_t)}\right)^{\frac{1}{n+1}}h\left(u,
t+\left(\frac{V(B)}{V(K_t)}\right)^{\frac{1+n-p}{n+1}}\tau\right).
\end{align*}
Note that $\bar{h}(\cdot,0)$ is the support function of $\left(V(B)/V(K_t)\right)^{\frac{1}{n+1}}K_t$; therefore,
\[r\leq \bar{h}(u,0)\leq R.\]
Write $\bar{K}_{\tau}$ for the convex body associated with $\bar{h}(\cdot,\tau)$ and let $B_c$ denote the ball of radius $c$ centered at the origin. Since $B_{R}$ encloses
$\bar{K}_0,$ the comparison principle implies that $B_{2R}$
will enclose $\bar{K}_{\tau}$ for $\tau\in[0,\delta],$
where $\delta$ depends only on $p,R,\psi$. By the first statement of Lemma \ref{lem: upper} applied to $\bar{h}$, there is a uniform lower bound (depending only on $r,R,p,\varphi$) on the Gauss curvature of $\bar{K}_{\frac{\delta}{2}}.$

On the other hand, the volume of $\bar{K}_{\frac{\delta}{2}}$ is bounded above by $V(B_{2R});$ therefore,
\[\displaystyle\frac{V(B)}{V(B_{2R})}\leq c_t:=\frac{V(K_t)}{V\left(K_{t+\left(\frac{V(B)}{V(K_t)}\right)^{\frac{1+n-p}{n+1}}\frac{\delta}{2}}\right)}\leq 1\]
for all $t\in [0,T)$. Consequently,
\begin{align*}
\left(\frac{V(B)}{V\left(K_{t+\left(\frac{V(B)}{V(K_t)}\right)^{\frac{1+n-p}{n+1}}\frac{\delta}{2}}\right)}\right)^{\frac{1}{n+1}}h\left(u,
t+\left(\frac{V(B)}{V(K_t)}\right)^{\frac{1+n-p}{n+1}}\frac{\delta}{2}\right)
=c_t^{\frac{1}{n+1}}\bar{h}\left(\cdot,\frac{\delta}{2}\right)
\end{align*}
has Gauss curvature bounded below for all $t\in [0,T)$.

Now we show that for every
 $\tilde{t}\in\left[\left(V(B)/V(K_0)\right)^{\frac{1+n-p}{n+1}}\frac{\delta}{2},T\right)$, we can find $t\in[0,T)$ such that
\[\tilde{t}=t+\left(\frac{V(B)}{V(K_t)}\right)^{\frac{1+n-p}{n+1}}\frac{\delta}{2}.\]
Define $f(t)=t+\left(\frac{V(B)}{V(K_t)}\right)^{\frac{1+n-p}{n+1}}\frac{\delta}{2}-\tilde{t}$ on $[0,T)$.
$f$ is continuous, and
\[\left\{
    \begin{array}{ll}
      f(T)=T-\tilde{t}>0, & p<n+1 \\
      f(\infty)=\infty, & p= n+1 \\
      f(0)\leq 0& p\leq  n+1 .
    \end{array}
  \right.
\]
The claim follows.

Next we obtain uniform lower and upper bounds on the principal curvatures of the normalized solution.

Consider the convex bodies $\tilde{K}_{\tau}:=\left(\frac{V(B)}{V(K_t)}\right)^{\frac{1}{n+1}}K_t,$ where
$$\tau(t):=\int_0^t \left(\frac{V(K_s)}{V(B)}\right)^{\frac{1+n-p}{n+1}}ds,\footnote{Suppose $p<n+1.$ For each $t\in [0,T)$ by the comparison principle we have
\[\frac{(\max h_{K_t})^{p-n-1}}{(n+1-p)\max \varphi}\leq T-t\leq \frac{(\min h_{K_t})^{{p-n-1}}}{(n+1-p)\min \varphi}.\]
Therefore, since $\frac{\max h_{K_t}}{\min h_{K_t}}\leq \frac{R}{r}$ (see (\ref{ratio})), we get
\begin{align*}
c_1(T-t)^{\frac{1}{p-n-1}}\leq \min h_{K_t} \leq \left(\frac{V(K_t)}{V(B)}\right)^{\frac{1}{n+1}}\leq \max h_{K_t}\leq c_2(T-t)^{\frac{1}{p-n-1}}.
\end{align*}
Thus $\lim_{t\to T}\tau(t)=\infty.$}$$
Let us furnish all geometric quantities associated with $\tilde{K}_{\tau}$ by an over-tilde.
The evolution equation of $\tilde{h}_{\tau}$ is given by
\[\partial_\tau \tilde{h}_{\tau}=\varphi \tilde{h}^{2-p}\tilde{S}_n-\frac{\int_{\mathbb{S}^{n}}\varphi \tilde{h}^{2-p}\tilde{S}_n^2d\sigma}{(n+1)V(B)}\tilde{h}.\]
Since $\frac{\int_{\mathbb{S}^{n}}\varphi \tilde{h}^{2-p}\tilde{S}_n^2d\sigma}{(n+1)V(B)}$ is uniformly bounded above,
applying the maximum principle to
$\Theta=\frac 12\log(\|\tilde{W}\|^2)-\alpha\log \tilde{h},$
and arguing as in the proof of Lemma \ref{lem: final}, we see that $\|\tilde{W}\|$ has a uniform upper bound. This in turn, in view of our lower and upper bounds on the Gauss curvature of $\tilde{K}_{\tau}$, implies that we have uniform lower and upper bounds on the principal curvatures of $\tilde{K}_{\tau}$. Higher order regularity estimates and convergence in $C^{\infty}$ for a subsequence of $\{\tilde{K}_{\tau}\}$ follow from Krylov-Safonov \cite{Krylov-Safonov}, standard parabolic theory and the Arzel\`{a}-Ascoli theorem. The convergence for the full sequence when $p\geq 1$ follows from the uniqueness of the self-similar solutions to (\ref{def: self similar}); see \cite{Lu1,39}. Moreover, note that when $\varphi\equiv1$ and $-n-1<p<1$, by the result of \cite{Choi}, the limit is the unit sphere.
\section{Appendix}\label{appendix}
\subsection*{Evolution of polar bodies}
Let $K$ be a smooth, strictly convex body with the origin in its interior. Suppose $\partial K$, the boundary of $K,$ is parameterized by the radial function $r=r(u):\mathbb{S}^n\to \mathbb{R}.$
The metric $[g_{ij}]$, unit normal $\nu$, support function $h$, and the second fundamental form $[w_{ij}]$ of $\partial K$
can be written in terms of $r$ and its partial derivatives as follows:
\begin{description}
  \item[a] $ g_{ij}=r^2\bar{g}_{ij}+\bar{\nabla}_ir\bar{\nabla}_jr,$
  \item[b] $ \nu=\frac{r u-\bar{\nabla}r}{\sqrt{r^2+\|\bar{\nabla}r\|^2}},$
  \item[c] $ h=\frac{r^2}{\sqrt{r^2+\|\bar{\nabla}r\|^2}},$
  \item[d] $ w_{ij}=\frac{-r\bar{\nabla}^{2}_{ij}r+2\bar{\nabla}_ir\bar{\nabla}_jr+
  r^2\bar{g}_{ij}}{\sqrt{r^2+\|\bar{\nabla}r\|^2}}.$
\end{description}
Since $\frac{1}{r}$ is the support function of $K^{\ast}$ (see, e.g., \cite[page 57]{Schneider}), we can calculate the entries of $[\mathfrak{r}^{\ast}_{ij}]$:
\[\mathfrak{r}^{\ast}_{ij}=\bar{\nabla}^{2}_{ij}\frac{1}{r}+\frac{1}{r}\bar{g}_{ij}=
\frac{-r\bar{\nabla}^2_{ij}r+2\bar{\nabla}_ir\bar{\nabla}_jr+r^2\bar{g}_{ij}}{r^3}.\]
Thus, using (\textbf{d}) we get
\begin{align*}
\mathfrak{r}^{\ast}_{ij}=\frac{\sqrt{r^2+\|\bar{\nabla}r\|^2}}{r^3}w_{ij}.
\end{align*}
\begin{lemma}\label{app1}
As $K_t$ evolve by (\ref{eq: flow4}), their polars $K_t^{\ast}$ evolve as follows:
\[\partial_th^{\ast}=-\varphi\left(\frac{h^{\ast}u+\bar{\nabla} h^{\ast}}{\sqrt{h^{\ast2}+|\bar{\nabla} h^{\ast}|^2}}\right)\frac{(h^{\ast2}+|\bar{\nabla} h^{\ast}|^2)^{\frac{n+1+p}{2}}}{h^{\ast n+1}S_n^{\ast}},~~h^{\ast}(\cdot,t):=h_{K_t^{\ast}}(\cdot).\]
\end{lemma}
\begin{proof}
To obtain the evolution equation of $h_{K_t^{\ast}}$, we first need to parameterize $M_t$ over the unit sphere
\[F=r(u(\cdot,t),t)u(\cdot,t):\mathbb{S}^{n}\to\mathbb{R}^{n+1},\]
where $r(u(\cdot,t),t)$ is the radial function of $M_t$ in the direction $u(\cdot,t).$
Note that
\[\partial_t r=\varphi \frac{h^{2-p}}{\mathcal{K}}\frac{\sqrt{r^2+\|\bar{\nabla}r\|^2}}{r},\]
and
\[\mathcal{K}=\frac{\det w_{ij}}{\det g_{ij}},\quad \frac{1}{S_n^{\ast}}=\frac{\det \bar{g}_{ij}}{\det \mathfrak{r}_{ij}^{\ast}},\quad \frac{\det \bar{g}_{ij}}{\det g_{ij}}=\frac{1}{r^{2n-2}(r^2+\|\bar{\nabla}r\|^2)},\]
\[\quad h=\frac{1}{\sqrt{h^{\ast 2}+\|\bar{\nabla}h^{\ast}\|^2}}.\]
Now we calculate
\begin{align*}
\partial_t h^{\ast}&=\partial_t \frac{1}{r}\\
&=- \frac{h^{2-p}}{\mathcal{K}}\frac{\sqrt{r^2+\|\bar{\nabla}r\|^2}}{r^3}\varphi(\nu)\\
&=-h^{2-p}\frac{\sqrt{r^2+\|\bar{\nabla}r\|^2}}{r^3}\frac{\det g_{ij}}{\det w_{ij}}\varphi(\nu)\\
&=-h^{2-p}\frac{\sqrt{r^2+\|\bar{\nabla}r\|^2}}{r^3}\frac{\det \bar{g}_{ij}}{\det \mathfrak{r}_{ij}^{\ast}}\frac{\det g_{ij}}{\det \bar{g}_{ij}}\frac{\det \mathfrak{r}_{ij}^{\ast}}{\det w_{ij}}\varphi(\nu)\\
&=-\left(\frac{\sqrt{r^2+\|\bar{\nabla}r\|^2}}{r^3}\right)^{n+1}\frac{r^{2n-2}(r^2+\|\bar{\nabla}r\|^2)}{(h^{\ast 2}+\|\bar{\nabla}h^{\ast}\|^2)^{\frac{2-p}{2}}}\frac{\varphi(\nu)}{S_n^{\ast}}.
\end{align*}
Replacing $r$ by $1/h^{\ast}$ and taking into account (\textbf{b}) finishes the proof.
\end{proof}

\subsection*{Estimates for curvature derivatives}

For convenience we present some of the main ideas, how one can prove the alternative in Lemma \ref{lem: final} about balancing the curvature derivatives. This method was used in \cite{Guan} for a similar stationary prescribed curvature equation. Recall that
\[A_i=\frac{2-\varepsilon}{\|W\|^2\mathcal{K}} w_{ii}(\mathcal{K}_{;i})^2-\frac{w_{ii}}{\|W\|^2}\sum_{p,q}\mathcal{K}^{pp,qq}w_{pp;i}w_{qq;i},\]

\[B_i=\frac{2}{\|W\|^2}\sum_{j}w_{jj}\mathcal{K}^{jj,ii}w^2_{jj;i},\quad C_i=\frac{2}{\|W\|^2}\sum_{j\neq i}\mathcal{K}^{jj}w_{jj;i}^2,\]

\[D_i=\frac{1}{\|W\|^2}\mathcal{K}^{ii}\sum_j w_{jj;i}^2,\quad
E_i=\frac{2}{\|W\|^4}\mathcal{K}^{ii}\left(\sum_j w_{jj}w_{jj;i}\right)^2.\]

Note that the term $A_i$ looks slightly different from the term $A_i$ in \cite[p. 1309]{Guan}, where the $\mathcal{K}$ is not present in the denominator. We have to define $A_i$ in the way we did, because due to the inverse nature of the curvature flow equation we obtain an extra good derivative term. This allows us to choose the constant in $A_i$ as $2 - \varepsilon$, whereas a large constant was required in \cite{Guan} (denoted by $K$ there). Fortunately the proofs of \cite[Lemma~4.2, Lemma~4.3]{Guan} also work for sufficiently small $\varepsilon$. The remaining terms $B_i,C_i,D_i,E_i$ are all identical to those in \cite{Guan}.

In the following $\sigma_k$ denotes the $k$-th elementary symmetric function of principal curvatures. We begin by recalling the following special case ($k=n$) of inequality (2.4) from \cite[Lemma 2.2]{Guan}, which can be deduced easily by differentiating
\[ G=\left(\frac{\sigma_n}{\sigma_l}\right)^{\frac{1}{n-l}}\] twice, using the concavity of $G$ and applying the Schwarz inequality.
For any $\delta>0$, $1\leq i\leq n$ and $1\leq l<n$ we have
\begin{align*}
-&\mathcal{K}^{pp,qq}w_{pp;i}w_{qq;i}+\left(1-\frac{1}{n-l}+\frac{1}{(n-l)\delta}\right)\frac{(\mathcal{K}_{;i})^2}{\mathcal{K}}\geq \\
&\left(1+\frac{1-\delta}{n-l}\right)\frac{\mathcal{K}((\sigma_l)_{;i})^2}{\sigma_l^2}-\frac{\mathcal{K}}{\sigma_l}\sigma_l^{pp,qq}w_{pp;i}w_{qq;i}.
\end{align*}
In particular, by taking $\delta=\frac{1}{2-\varepsilon}$, we have
\begin{align}\begin{split}\label{key ineq}
(2-\varepsilon)\frac{(\mathcal{K}_{;i})^2}{\mathcal{K}}-\mathcal{K}^{pp,qq}w_{pp;i}w_{qq;i}&\geq \left[1+\frac{1-\varepsilon}{(n-1)(2-\varepsilon)}\right]\frac{\mathcal{K}((\sigma_l)_{;i})^2}{\sigma_l^2}\\
					&\hphantom{=}-\frac{\mathcal{K}\sigma_l^{pp,qq}w_{pp;i}w_{qq;i}}{\sigma_l},
\end{split}\end{align}
provided $(2-\varepsilon) > 1$, i.e. $0 < \varepsilon < 1$.

\begin{lemma}\label{app2}
For each $i\ne 1,$ if $\sqrt{3}\kappa_i\leq \kappa_1,$ we have
\[A_i+B_i+C_i+D_i-E_i\geq 0.\]
\end{lemma}
\begin{proof}
Note that from (\ref{key ineq}) with $l=1$, it follows that $A_i\geq 0$ since $\sigma_1^{pp,qq} = 0$. The proof of that $B_i+C_i+D_i-E_i\geq 0$ can literally be taken from \cite[Lemma 4.2]{Guan}, starting with \cite[Equ.~(4.10)]{Guan}.
\end{proof}

In the following proof we will write $\sigma_n=\mathcal{K}$ for a better comparability with \cite[Lemma~4.3]{Guan}. Also denote by $\sigma_k(\kappa|i)$ the $k$-th elementary symmetric polynomial in the variables $\kappa_1,\dots,\kappa_{i-1},\kappa_{i+1},\dots,\kappa_n$ and $\sigma_k(\kappa|ij)$ accordingly.

\begin{lemma}\label{app3}
For $\lambda=1,\ldots,n-1$ suppose there exists some $\delta\leq 1$ such that $\kappa_{\lambda}/\kappa_1\geq \delta.$ There exists a sufficiently small positive constant $\delta'$ depending on $\delta$, $\epsilon$ and the bounds for $\mathcal{K}$, such that if $\kappa_{\lambda+1}/\kappa_1\leq \delta',$ we have
\[A_i+B_i+C_i+D_i-E_i\geq 0\quad \text{for}~i=1,\ldots,\lambda.\]
\end{lemma}
\begin{proof}
This corresponds to \cite[Lemma~4.3]{Guan}. We highlight the main estimates in this proof. First of all, from \cite[Equ.~(4.16), (4.17)]{Guan} one can extract the following estimate:
\begin{align}\begin{split}\label{app3-1}
\|W\|^4(B_i+C_i+D_i-E_i)&\geq \|W\|^2\sum_{j\neq i} \left(\sigma_{n-1}(\kappa|j)-2\sigma_{n-1}(\kappa|ij)\right)w_{jj;i}^2\\
				&\hphantom{=}-w_{ii}^2\sigma_{n}^{ii}w_{ii;i}^2 \\
&= \|W\|^2\sum_{j\neq i} \sigma_{n-1}(\kappa|j) w_{jj;i}^2 - w_{ii}^2\sigma_{n}^{ii}w_{ii;i}^2, \\
\end{split}\end{align}
since $\sigma_{n-1}(\kappa|ij) = 0$.

Now we show the right hand side of \eqref{app3-1} is dominated by $\|W\|^4A_i$. From \eqref{key ineq} we get for all $1\leq \lambda<n$ and for all $1\leq i\leq n\colon$
\begin{align}\begin{split}\label{}
A_i&=\frac{(2-\varepsilon) w_{ii}}{\|W\|^2\sigma_n}((\sigma_n)_{;i})^2-\frac{w_{ii}}{\|W\|^2}\sum_{p,q}\sigma_n^{pp,qq}w_{pp;i}w_{qq;i}\\
		&\geq \frac{w_{ii}}{\|W\|^2}\left(1+\frac{1-\varepsilon}{(n-1)(2-\varepsilon)}\right)\frac{\sigma_n((\sigma_{\lambda})_{;i})^2}{\sigma_{\lambda}^2}\\
					&\hphantom{=}-\frac{w_{ii}}{\|W\|^2}\frac{\sigma_n\sum_{p,q}\sigma_{\lambda}^{pp,qq}w_{pp;i}w_{qq;i}}{\sigma_{\lambda}}\\
                    &=\frac{w_{ii}\sigma_n}{\|W\|^2\sigma_{\lambda}^2}\Big[\left(1+\frac{1-\varepsilon}{(n-1)(2-\varepsilon)}\right)\sum_{a}\left(\sigma_{\lambda}^{aa}w_{aa;i}\right)^2\\
                    	&\hphantom{\frac{w_{ii}\sigma_n}{\|W\|^2\sigma_{\lambda}^2}\Big(}+\frac{1-\varepsilon}{(n-1)(2-\varepsilon)}\sum_{a\neq b}\sigma_\lambda^{aa}\sigma_{\lambda}^{bb}w_{aa;i}w_{bb;i}\\
                        &\hphantom{\frac{w_{ii}\sigma_n}{\|W\|^2\sigma_{\lambda}^2}\Big(}+\sum_{a\neq b}\left(\sigma_{\lambda}^{aa}\sigma_{\lambda}^{bb}-\sigma_{\lambda}\sigma_{\lambda}^{aa,bb}\right)w_{aa;i}w_{bb;i}\Big].
\end{split}\end{align}

For sufficiently small $\delta'$ and $\lambda=1$ the simple estimates \cite[Equ.~(4.19), (4.20)]{Guan} give
\begin{equation}
\|W\|^4A_i\geq w_{ii}^2\sigma_n^{ii}w^{2}_{11;i}-C_{\epsilon}w_{ii}\sum_{a\neq 1}w_{aa;i}^2.
\end{equation}

Combining this with \eqref{app3-1} for $i=1$ yields,
\begin{align}\begin{split}\label{app3-2}
\|W\|^2(A_1+B_1+C_1+D_1-E_1)&\geq \sum_{j\neq 1} \sigma_{n-1}(\kappa|j)w_{jj;1}^2-\frac{C_{\epsilon}}{w_{11}}\sum_{j\neq 1}w_{jj;1}^2\\
		&=\sum_{j\neq 1}\left(\frac{\sigma_n}{w_{jj}}-\frac{C_{\epsilon}}{w_{11}}\right)w_{jj;1}^2\\
        &\geq \sum_{j\neq 1}\left(\frac{\sigma_n}{\delta' w_{11}}-\frac{C_{\epsilon}}{w_{11}}\right)w_{jj;1}^2,
\end{split}\end{align}
which is non-negative for $\delta'$ sufficiently small. Hence the lemma is true in the case $\lambda=1.$

For $\lambda>1$ the series of elementary estimates \cite[Equ.~(4.22)-(4.27)]{Guan} gives
\[
\|W\|^4A_i\geq w_{ii}^2\sigma_n^{ii}\sum_{a\leq\lambda}w_{aa;i}^2-\frac{w_{ii}C_{\epsilon}}{\delta^2}\sum_{a>\lambda}w_{aa;i}^2,
\]
after having adapted $\epsilon$ if necessary and having chosen $\delta'$ sufficiently small
again. Combining this last inequality with \eqref{app3-1} for $1\leq i\leq \lambda$ yields
\begin{align}\begin{split}
\|W\|^2(A_i+B_i+C_i+D_i-E_i)&\geq \sum_{j\neq i}\sigma_{n-1}(\kappa|j)w_{jj;i}^2-\frac{C_{\epsilon}}{w_{ii}\delta^2}\sum_{j>\lambda}w_{jj;i}^2\\
			&\geq \sum_{j>\lambda}\left(\sigma_{n-1}(\kappa|j)-\frac{C_{\epsilon}}{w_{ii}\delta^2}\right)w_{jj;i}^2\\
                        &\geq \sum_{j>\lambda}\left(\frac{\sigma_n}{w_{11}\delta'}-\frac{C_{\epsilon}}{w_{ii}\delta^2}\right)w_{jj;i}^2,
\end{split}\end{align}
which is non-negative for small $\delta'$ for the same reason as in \eqref{app3-2}. This completes the proof.
\end{proof}

\begin{corollary}\label{Alternative}
There exist positive numbers $\delta_2,\dots,\delta_n,$ depending only on the dimension, on $\epsilon$ and on the bounds for the Gauss curvature, such that either
\begin{equation}
\kappa_i > \delta_i\kappa_1\quad\forall 2\leq i\leq n
\end{equation}
or
\begin{equation}
A_i+B_i+C_i+D_i-E_i\geq 0\quad\forall 1\leq i\leq n.
\end{equation}
\end{corollary}
\begin{proof}
Choosing $\lambda=1$ and $\delta=1$ in Lemma \ref{app3} yields the existence of $\delta'$ with the following property: if $\kappa_2/\kappa_1\leq \delta'$, then
\[A_1+B_1+C_1+D_1-E_1\geq 0.\]
Note that $\kappa_i\leq \kappa_2$ for $i\geq 2.$ Choose $\delta_2=\min\{\delta',1/\sqrt{3}\}$. Therefore, in view of Lemma \ref{app2}, $\kappa_2/\kappa_1\leq \delta_2$ implies that \[A_i+B_i+C_i+D_i-E_i\geq 0\quad \forall i\geq 2.\]

We now apply induction, assuming we have constructed $\delta_2,\ldots, \delta_j$. We may assume $\kappa_i > \delta_i \kappa_1$ for $2\leq i \leq j$ otherwise $A_i+B_i+C_i+D_i-E_i\geq 0$ is already true for $2\leq i \leq n$. Choose $\delta = \delta_j$ and $\lambda=j$ in Lemma \ref{app3} to get a $\delta'$ so that if $\kappa_{j+1} \leq \delta' \kappa_1$, then  $A_i+B_i+C_i+D_i-E_i\geq 0$ holds for $1\leq i \leq j$. Now in view of Lemma \ref{app2}, taking $\delta_{j+1} = \min \{\delta', 1/\sqrt{3}\}$ gives $A_i+B_i+C_i+D_i-E_i\geq 0$ for $j \leq i \leq n$.
\end{proof}
\bibliographystyle{amsplain}

\end{document}